\documentclass[a4paper,USenglish]{lipics-v2021}
\hideLIPIcs
\nolinenumbers

\usepackage{amsthm}
\usepackage{amsmath}
\usepackage{amssymb}
\usepackage{amsfonts}
\usepackage{epsfig}
\usepackage{graphics}
\usepackage{graphicx}
\usepackage{booktabs}
\usepackage{pdfpages}
\usepackage{tikz}
\usepackage{pgfplots}
\pgfplotsset{compat=1.18}
\usepackage[T1]{fontenc}

\usepackage{complexity}
\usepackage{url}

\usepackage{latexsym}
\usepackage{psfrag}
\usepackage{enumerate}
\usepackage{here}
\usepackage{esvect}
\usepackage[update,prepend]{epstopdf}
\usepackage{anyfontsize}
\usepackage{url}
\usepackage{hyperref}

\newcommand{\ie}{{i.e.}}
\newcommand{\eg}{{e.g.}}

\newcommand{\area}{{\rm Area}}

\newcommand{\diam}{{\rm diam}}

\newtheorem{problem}[theorem]{Problem}

\def\P{\mathcal P}

\def\V{\mathcal V}

\usepackage[noline,linesnumbered]{algorithm2e}
\SetArgSty{textrm}
\let\oldnl\nl
\newcommand{\nonl}{\renewcommand{\nl}{\let\nl\oldnl}}
\makeatletter
\def\TitleOfAlgo{\@ifnextchar({\@TitleOfAlgoAndComment}{\@TitleOfAlgoNoComment}}
\def\@TitleOfAlgoAndComment(#1)#2{\nonl\hspace*{-1.5em}#2 #1\;}
\def\@TitleOfAlgoNoComment#1{\nonl\hspace*{-1.5em}#1\;}
\makeatother
\DontPrintSemicolon

\newcommand{\later}[1]{}
\newcommand{\old}[1]{}

\title{Covering Complete Geometric Graphs by Monotone Paths}

\titlerunning{Covering Complete Geometric Graphs by Monotone Paths}

\author{Adrian Dumitrescu}
{Algoresearch L.L.C., Milwaukee, WI, USA}
{ad.dumitrescu@algoresearch.org}
{0000-0002-1118-0321}{}
\author{J\'anos Pach}
{Alfr\'ed R\'enyi Institute of Mathematics, Budapest, Hungary}
{pach@renyi.hu}
{0000-0002-2389-2035}
{}
\author{Morteza Saghafian}
{IST Austria (Institute of Science and Technology Austria), Kloster\-neu\-burg, Austria}
{morteza.saghafian@ist.ac.at}
{0000-0002-4201-5775}
{}
\author{Alex Scott}
{Mathematical Institute, University of Oxford, UK}
{scott@maths.ox.ac.uk}
{0000-0003-4489-5988}
{}

\authorrunning{Adrian Dumitrescu, J\'anos Pach, Morteza Saghafian, and Alex Scott}

\ccsdesc[500]{Mathematics of computing~Discrete mathematics}
\ccsdesc[500]{Theory of Computation~Randomness, geometry and discrete structures}

\keywords{convexity, geometric graph, complete graph, crossing family, plane subgraph}

\funding{Research partially supported by ERC Advanced Grant "GeoScape",
  no. 882971 and Hungarian NKFIH grant no. K-131529.
  Work by the third author is supported by EPSRC grant EP/X013642/1.
  Work by the third author is partially supported by the European Research
  Council (ERC), grant no.\ 788183, and by the Wittgenstein Prize, 
  Austrian Science Fund (FWF), grant no.\ Z~342-N31.}

\begin{document}

\maketitle

\begin{abstract}
  Given a set $A$ of $n$ points (vertices) in general position in the plane, the
  \emph{complete geometric graph} $K_n[A]$ 
  consists of all $\binom{n}{2}$ segments (edges) between the elements of $A$.
  It is known that the edge set of every complete geometric graph on $n$
  vertices can be partitioned into $O(n^{3/2})$ crossing-free paths (or matchings).
  We strengthen this result under various additional assumptions on the  point set.
  In particular, we prove that for a set $A$ of $n$ \emph{randomly} selected points,
  uniformly distributed in $[0,1]^2$, with probability tending to $1$ as $n\rightarrow\infty$,
  the edge set of $K_n[A]$ can be covered by $O(n\log n)$ crossing-free paths and by 
  $O(n\sqrt{\log n})$ crossing-free matchings.
On the other hand, we construct $n$-element point sets such that covering the edge set of $K_n[A]$ requires a quadratic number of monotone paths.  
\end{abstract}

\section{Introduction} \label{sec:intro}

A set of points in the plane is said to be
(i)~in \emph{general position} if no $3$ points are collinear; and
(ii)~in \emph{convex position} if every point in the set is an extreme point of the convex hull.
Given $n$ points in general position in the plane, the graph obtained 
by connecting certain point-pairs by straight-line segments is called a \emph{geometric graph} $G$.
If no two segments (edges) of $G$ cross each other, then $G$ is said to be
\emph{crossing-free} or a \emph{plane graph}. 

In $2005$, Araujo, Dumitrescu, Hurtado, Noy, and Urrutia~\cite{ADH+05} asked the following question:
Can the edge set of every complete geometric graph on $n$ vertices be partitioned into a small number
of crossing-free matchings? In particular, they asked:

\begin{problem} \label{problem:1}
  Does there exist a constant $c>0$ such that every complete geometric graph on $n$ vertices
  can be partitioned into at most $c n$ plane matchings?
\end{problem}

It is easy to verify that when $n$ is even (respectively odd), every complete
geometric graph of $n$ vertices in convex position 
can be decomposed into $n-1$ (respectively $n$) perfect plane matchings.
On the other hand, the best known upper bound for point sets in general position is only
$O(n^{3/2})$~\cite{ADH+05}, which seems to be far off.

Given a point-set $A$, let $K_n[A]$ denote the complete geometric graph induced by $A$.
We prove the following.

\begin{theorem} \label{thm:random}
  Let $A$ be a set of $n$ random points uniformly distributed in $[0,1]^2$, and let $n\rightarrow\infty$.  
  Then, with probability tending to $1$, the edge set of $K_n[A]$ can be covered by at most
  $O(n\log n)$ crossing-free paths, and by $O(n\sqrt{\log n})$ crossing-free matchings. 
\end{theorem}

This is better than the $O(n^{3/2})$ bound in~\cite{ADH+05},
which holds for every point set in general position.  The proof in fact gives the stronger result 
that there is a covering by $O(n\log n)$ monotone paths.

\begin{definition}
A polygonal path $\xi=(v_1,v_2,\ldots, v_t)$ in $\mathbb{R}^2$ is
\emph{monotone in direction $\mathbf{u}\in \mathbb{R}^2\setminus \{\mathbf{0}\}$}
(or is $\mathbf{u}$-\emph{monotone})
if every directed edge of $\xi$ has a positive inner product with $\mathbf{u}$, that is,
$\langle \overrightarrow{v_iv_{i+1}},\mathbf{u}\rangle>0$ for $i=1,\ldots, t-1$;
here $\mathbf{0}$ is the origin.
A path $\xi=(v_1,v_2,\ldots, v_t)$ is \emph{monotone}
if it is monotone in some direction $\mathbf{u}\in \mathbb{R}^2\setminus \{\mathbf{0}\}$. 
Obviously, every monotone path is crossing-free.

A matching is said to be \emph{monotone in direction $\mathbf{u}\in \mathbb{R}^2\setminus \{\mathbf{0}\}$}
if the edge projections onto $\mathbf{u}$ are disjoint intervals.
A matching is said to be \emph{monotone}
if it is monotone in some direction $\mathbf{u}\in \mathbb{R}^2\setminus \{\mathbf{0}\}$.
Equivalently, a matching is monotone if its edge set is the subset of the edge set of a monotone path.

For any constant $\alpha>0$, a set of $n$ points is \emph{$\alpha$-dense} if the ratio
between the longest and the shortest distance between two points in the set is at most $\alpha \sqrt{n}$.
\end{definition}

For dense sets, the $O(n^{3/2})$ upper bound~\cite{ADH+05} for decompositions into
plane matchings can be strengthened by requiring that the paths are \emph{monotone}.

\begin{theorem}\label{thm:dense}
  Let $A$ be an $\alpha$-dense point set in general position in the plane. 
  Then the edge set of $K_n[A]$ can be covered by $O(n^{3/2})$ {monotone} paths, 
  and such a covering can be computed in $O(n^2)$ time.
  Consequently, the edge set of $K_n[A]$ can be decomposed into $O(n^{3/2})$ \emph{monotone} matchings. 
\end{theorem}

From the opposite direction, we suspect that covering the edge set of the complete geometric graph
of every dense point set requires a superlinear number of monotone paths.

\begin{conjecture}\label{conj:dense-lb}
  For some (possibly, for every) dense $n$-element point set $A$ in general position in the plane,
  covering the edge set of $K_n[A]$ requires a superlinear number of monotone paths.
\end{conjecture}

In contrast to random or dense point sets, arbitrary point sets require a quadratic number
of monotone paths in the worst-case.

\begin{theorem}\label{thm:arbitrary}
  Let $A$ be a set of $n$ points in general position in the plane.
  Then the edge set of the complete geometric graph $K_n[A]$ can be covered by $n^2/6 + O(n)$ monotone paths. 
  On the other hand, there exist $n$-element point sets that require at least $n^2/15$ monotone paths.
\end{theorem}

Monotone paths are ubiquitous and have been studied extensively over the years,
particularly in the field of optimization~\cite{DRT13}. On the other hand,
crossing-free paths and matchings have been an attraction for geometers~\cite{KPT97,KPTV98,KLN91}.
Note that every monotone path is crossing-free. Clearly, by taking the odd, resp. even numbered edges, every monotone path decomposes into two monotone matchings. 

The rest of this paper is organized as follows. In Sections~2, 3, and 4, we prove Theorems~2, 3, and~5,
respectively. The last section contains some concluding remarks and open problems.  
All point sets appearing in this paper are assumed to be in general position, and the logarithms are in base~$2$.

\section{Proof of Theorem~\ref{thm:random}} 

Let $A=a_1,\ldots,a_n$ be a random sequence of $n$ points, independently and uniformly distributed in $U=[0,1]^2$.
For any unit vector $\mathbf{v} \in {\mathbb R}^2$, we can sort $A$ according to the projections
 $\langle a_i, \mathbf{v}\rangle$.
Observe that sorting with respect to $\mathbf{v}$ and $-\mathbf{v}$ are equivalent. 
Let $\theta=\alpha(\log n) / n$ be a small angle, where $\alpha>0$ is a suitable small constant,
so that $N=\pi/\theta$ is an integer.
Let $\V$ be a set of $N$ unit vectors evenly spaced in $[0,\pi)$ so that $(1,0) \in \V$.  

Consider the following procedure for covering the edge set of $K_n[A]$ by (i) monotone paths
and (ii) crossing-free matchings. For each vector $\mathbf{v} \in \V$, 
sort the elements of $A$ according to their projection on $\mathbf{v}$, and label them in increasing order as  $a_1^\mathbf{v},\dots,a_n^\mathbf{v}$. Now proceed as follows:

\begin{itemize}
\item[$\bullet$] To produce a collection of monotone paths: for each $\mathbf{v} \in \V$, and $1\le i \le j\le \beta\log n$,
we consider the sequence $(a_k^\mathbf{v}:k\equiv i\bmod j)$. This defines a monotone path for each choice of $i$, $j$ and
$\mathbf{v}$.
\item[$\bullet$] To produce a collection of non-crossing matchings: we partition $A$ into 
$t:=\lceil n/(\beta\log n) \rceil$ intervals $B_1,B_2\dots B_t$ of size at most $\lceil\beta\log n\rceil$. 
For each set of the form $B_i\cup B_{i+1}$, by the $O(n^{3/2})$ aforementioned bound in~\cite{ADH+05},
we can cover the complete graph on  $B_i\cup B_{i+1}$ with at most $(2\lceil\beta\log n\rceil)^{3/2}$ matchings. 
We use these to produce a total of $2(2\lceil\beta\log n\rceil)^{3/2}$ matchings: first pair up the matchings 
from $B_1\cup B_2, B_3\cup B_4, \dots$, and then pair up the matchings from $B_2\cup B_3, B_4\cup B_5, \dots$.
\end{itemize}

For each choice of $\mathbf{v}$, the first bullet produces a set of at most $(\beta^2/2)\log^2n$ paths; 
and the second bullet produces a set of at most $2(2\lceil\beta\log n\rceil)^{3/2}$ matchings.  
Since there are $N=n\pi/\alpha\log n$ choices of $\mathbf{v}$, we use a total of at most $O(n\log n)$ paths 
and at most $O(n\sqrt{\log n})$ matchings.

All that remains is to check that the paths and matchings we have produced do in fact cover all pairs of points from $A$.  
Note that, for each of the orderings we define, we have covered all pairs that are at most $\beta\log n$ apart in the ordering.
Thus it is enough to show that, for every pair of points $a,b\in A$, there is some $\mathbf{v}$ such that $a$ and $b$ 
are at most $\beta\log n$ apart in the corresponding ordering of $A$.

We first note a geometric fact.
Let $a,b \in A$ be any two (distinct) points in $U=[0,1]^2$.
We rotate the picture, and assume for convenience that $ab$
is a vertical segment, and that the unit square has been rotated by some angle. 
Let $U_{\gamma}$ denote the square rotated by $\gamma$; we have $\area(U_{\gamma})=1$ for every $\gamma$. 
Consider the geometric locus $R(a,b)$ of the points $p \in U_{\gamma}$ on lines making
an angle of at most $\theta$ with $ab$ and which intersect $ab$, see Fig.~\ref{fig:locus}.

\begin{figure}[htbp]
\centering
\includegraphics[scale=0.43]{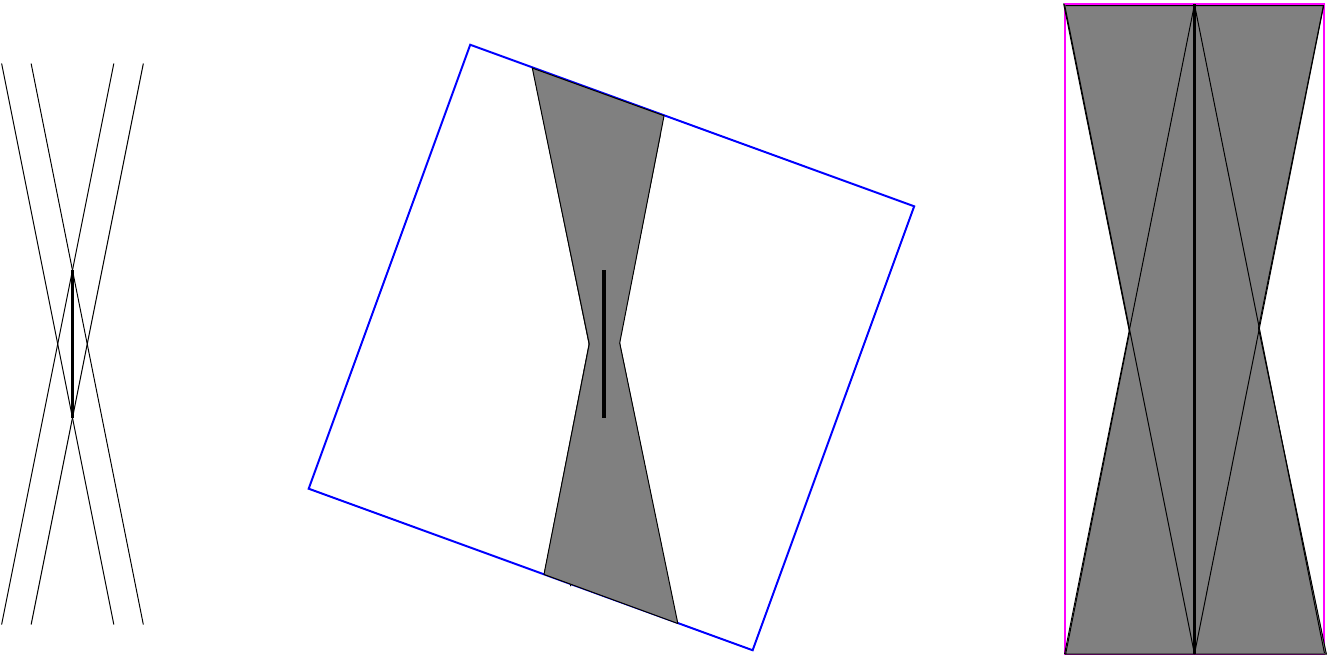}
\caption {Left: four lines making an angle of $\theta$ with $ab$ through $a$ and $b$.
  Center: the geometric locus (shaded) associated with a vertical segment.
  Right: a bounding rectangle for the locus associated with a long vertical segment.}
\label{fig:locus}
\end{figure}

\begin{lemma}  \label{lem:locus}
The area of $R(a,b)$ is at most $4 \theta$. 
\end{lemma}
\begin{proof}
  Consider four lines incident to $a$ and $b$ respectively, and making 
(clockwise or counterclockwise) angles of $\theta$ with $ab$.
  The boundary of the locus is made by these four lines and two polygonal arcs
  on the boundary of the square. The locus is easily seen to be contained in an axis-parallel rectangle
  of width $2 \sqrt{2} \tan \theta$ and height $\sqrt2$.
  As such, its area is bounded from above by $ 4 \tan \theta$.
  Excluding the two isosceles triangles based on the left and right vertical sides yields an improved
  area bound of $\frac34 \cdot 4 \tan \theta = 3 \tan \theta \leq 4 \theta$
  (here we assume that $n$ is sufficiently large). 
\end{proof}

We are now ready to complete the proof.  We reveal two points $a,b\in A$ (and note that $a\ne b$ with probability 1).
Choose a vector $\mathbf{v} \in \V$ that makes an angle at most $\theta$ with the direction orthogonal to $ab$.
Then any point of $A$ that comes between $a$ and $b$ in the ordering generated by $\mathbf{v}$ must 
lie in the region $R(a,b)$, and the lemma shows that, for any given point, the probability that the point lies in this region is at most $4\theta$.  The probability that at least $\beta\log n$ points of $A$ lie in $R(a,b)$ is therefore at most
\[
\binom{n}{\lceil \beta\log n\rceil}(4\theta)^{\lceil \beta\log n\rceil}
\le \left(\frac{en}{\lceil \beta\log n\rceil}\right)^{\lceil \beta\log n\rceil}(4\theta)^{\lceil \beta\log n\rceil}
\le \left(\frac{4\alpha e}{\beta}\right)^{\beta\log n},
\]
where we used the estimate $\binom n k\le(en/k)^k$.
This probability is $o(1/n^2)$, provided we have chosen $\alpha$ and $\beta$ sensibly.  There are $O(n^2)$ pairs: 
so, by the union bound, with high probability,  for every pair $a,b\in A$ there is a direction in which $a$ and $b$ 
are separated by at most $\beta\log n$ other points.  This completes the proof.
\qed

\section{Proof of Theorem~\ref{thm:dense}}

Consider an $\alpha$-dense point set $A$.
By scaling, we can assume that $A$ is contained in the unit square $U=[0,1]^2$,
and the distance between any two points of $A$ is at least $\sqrt{2} \cdot \alpha^{-1} n^{-1/2}$.

The covering procedure is similar to Phase I in the proof of Theorem~\ref{thm:random}, but with a twist.
Let $\varphi=c n^{-1/2}$ be a small angle, for a suitable small constant $c>0$, so that $N=2\pi/\varphi$ 
is an integer. Let $\V$ be a set of $N$ unit vectors evenly spaced around the origin so that 
$(1,0) \in \V$. For each vector $\mathbf{v} \in \V$, compute the $\mathbf{v}$-monotone spanning path
$\xi_v$ induced by $A$.
Let the path be $a_1a_2\ldots a_n$ in the order of monotonicity. Set 
\[ \beta=\frac{8 \alpha^2 c +4\sqrt{2} \cdot \alpha}{\pi}, \]
and for each $1\leq j \leq i \leq \beta\sqrt{n}$, 
add the path $a_ja_{i+j}a_{2i+j}\ldots$ to the set $\mathcal{P}$ of covering paths.
Note that for each vector $\mathbf{v} \in \V$, there are at most $1+2+\cdots+\beta\sqrt{n} < \beta^2n$
paths to be added to $\mathcal{P}$, all of which are monotone and, therefore, $\mathcal{P}$ contains 
at most $\frac{2\pi}{c} \sqrt{n} \cdot \beta^2n = \Theta(n^{3/2})$ monotone paths.

We claim that the paths in $\mathcal{P}$ cover the complete geometric graph on $A$. That is, 
for any pair of points $a,b \in A$, we show that the segment $ab$ appears in at least one path in $\P$. 
We construct the region $R(a,b)$ analogous to the one in the proof of Theorem~\ref{thm:random}
but with respect to $\varphi$ instead. By Lemma~\ref{lem:locus}, its area is at most $4\varphi$.
Assume that $R(a,b)$ contains $k$ points of $A$. The disks of radius 
$r=\frac{\sqrt{2}}{2} \cdot \alpha^{-1} n^{-1/2}$
centered at these $k$ points are disjoint and they cannot exceed the region $R(a,b)$ extended by a margin of $r$, whose area is then at most $4(\varphi+r)$. The total area of the $k$ disks is at most $4(\varphi+r)$,
and therefore
\begin{align*}
 k\pi r^2 &\leq 4(\varphi + r), \text{ thus }\\
k\pi &\leq 8 \alpha^2 n \left(\frac{c}{\sqrt{n}} + \frac{\sqrt{2}}{2 \alpha \sqrt{n}} \right)
\text{ or } \\
k &\leq \frac{8 \alpha^2 c +4\sqrt{2} \cdot \alpha}{\pi} \cdot \sqrt{n} = \beta\sqrt{n}.
\end{align*}

This means that in any $\mathbf{v}$-monotone spanning path induced by $A$,
where $\mathbf{v} \in \V$ makes an angle of at most $\varphi$ with $ab$, 
there are at most $\beta\sqrt{n}$ points between $a,b$ and thus, 
the segment $ab$ appears in one of the monotone paths constructed from $\xi_v$.

Since every monotone path can be decomposed into two monotone matchings 
(the odd- and even-numbered edges in the path),
the edge set of  $K_n[A]$ can be covered by $\Theta(n^{3/2})$ monotone matchings.

For the analysis of the algorithm, we have $\Theta(\sqrt{n})$ vectors and  for
each vector $v$ we first sort the points according to their $v$ component in
$O(n\log n)$ time. Finally, by tracing the points in  the sorted array,
each point may update at most $O(\sqrt{n})$ monotone paths corresponding to~$v$.
Overall, the running time of the algorithm is 
\[ \Theta(\sqrt{n}) \, \left[O(n \log n) + n \cdot O(\sqrt{n}) \right ] = O(n^2). \]
\qed

\section{Proof of Theorem~\ref{thm:arbitrary}}

\emph{Lower bound.}
Assume that $n$ is divisible by $3$, \ie, $n=3k$. 
Partition the $n$ points in our set $P$ into three groups $A,B,C$, of the same size, 
close to the vertices of a unit equilateral triangle, 
where each group consists of roughly equidistant points, but the distances are different;
see Fig.~\ref{fig:tripartite}. Let the diameters of the 3 groups of $k$ vertices be 
\[ \diam(C) \ll \diam(B) \ll \diam(A) \ll 1. \]

The number of (undirected) inter-group edges, namely edges in 
\[ E_0 := E(A\times B) \cup E(A\times C) \cup E(B\times C), \]
is $3k^2=n^2/3$.
Let $\P$ be a covering of $K_n[P]$ by \emph{monotone} paths. 
We claim that every path $\xi \in \P$ contains at most five edges in $E_0$. 
Consequently, covering all the edges of the tripartite graph on 
$P = A \cup B \cup C$  requires at least $3k^2/5 = n^2/15$ monotone paths. 

\begin{figure}[htbp]
\centering
\includegraphics[scale=0.55]{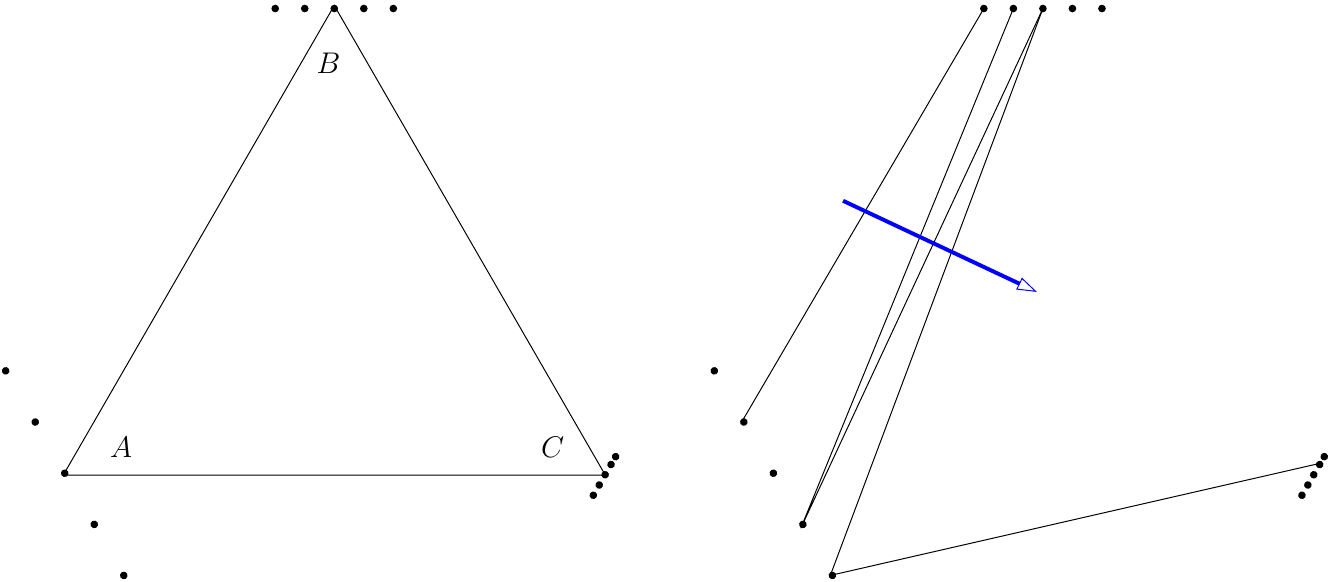}
\caption {Covering the edge set of a tripartite graph by monotone paths.
  The figure shows the edges in $E_0$ corresponding to a monotone path $\xi$;
  note that this edge set may be disconnected.
  A~direction of monotonicity is drawn in blue color.}
\label{fig:tripartite}
\end{figure}

To verify the claim, observe that the directed edges in $\xi$ can be of six types:

\begin{enumerate} \itemsep 2pt
\item $A \to B$, and the corresponding opposite orientation $B \to A$
\item $A \to C$, and the corresponding opposite orientation $C \to A$
\item $B \to C$, and the corresponding opposite orientation $C \to B$
\end{enumerate}

Let us trace the edges of $\xi$ from one end of the path to the other,
and record the sequence of groups in the order they are visited.
(When two consecutive vertices belong to the same group, we do not repeat the corresponding symbol in this sequence.) 

Due to the monotonicity of $\xi$, the resulting sequence satisfies the following conditions:

\begin{enumerate} [(i)]  \itemsep 2pt
\item There is no cycle of the form $A \to B \to C \to A$.
\item For any directed edge type $U$, let $U^r$ denote the type of an edge with opposite orientation.
  There are no two types, $U$ and $V$, such that
  $\xi$ has edges belonging to all four types, $U$, $U^r$, $V$, and $V^r$.
  This follows from the fact that there is no straight-line on which the orthogonal projection
  of four such edges would be disjoint. 
  \item For any two groups, $X,Y \in \{A,B,C\}$, there are at most \emph{four} edges 
    between $X$ and $Y$ in $\xi$. Moreover, if $\diam(X) \gg \diam(Y)$
    and there are four edges between $X$ and $Y$ in $\xi$, then $\xi$ starts or ends at
    a point in $X$.  
\end{enumerate}

This implies that there are no monotone paths consisting of more than \emph{five} edges in $E_0$,
and the only monotone paths covering five edges are of the following form:

\begin{enumerate} \itemsep 2pt
\item $A \to B \to A \to B \to A \to C$ (shown in Fig.~\ref{fig:tripartite}). 
\item $A \to C \to A \to C \to A \to B$
\item $B \to C \to B \to C \to B \to A$
\end{enumerate}

\medskip
\emph{Upper bound.}
Since every 2-edge path is monotone, one can decompose the edge set of $K_n[A]$ into
$\lceil \frac12 \binom{n}{2} \rceil$ 2-edge paths.
We can improve upon this bound by using $3$-edge paths. Since not every $3$-edge path is monotone,
we need special $3$-edge paths that are monotone.   
A~$3$-edge path is a \emph{zig-zag} path if its two endpoints lie in opposite open halfplanes
determined by the middle edge of the path. It is easy to see that any $3$-edge zig-zag path
is monotone.

The key is a result by Wilson \cite{Wil75} about decomposition of the edge set of complete graphs
into complete subgraphs of the same size. The case we are interested in is as follows.

\begin{lemma}[Wilson 1975]
  Let $n$ be a large enough positive integer. Then the edge set of $K_n$ can be decomposed into subgraphs
  of the form $K_6$, if and only if $n$ satisfies two divisibility conditions,
  namely $5\ |\ (n-1)$ and $\binom{6}{2} \ | \ \binom{n}{2}$.
  
  Moreover, if $n$ does not satisfy the above divisibility conditions, then one can still pack
  $n^2/30 - O(n)$ $K_6$'s into $K_n$, leaving out only linearly many edges. 
\end{lemma}

\begin{figure}[htbp]
\centering
  
\includegraphics[width=0.235\textwidth]{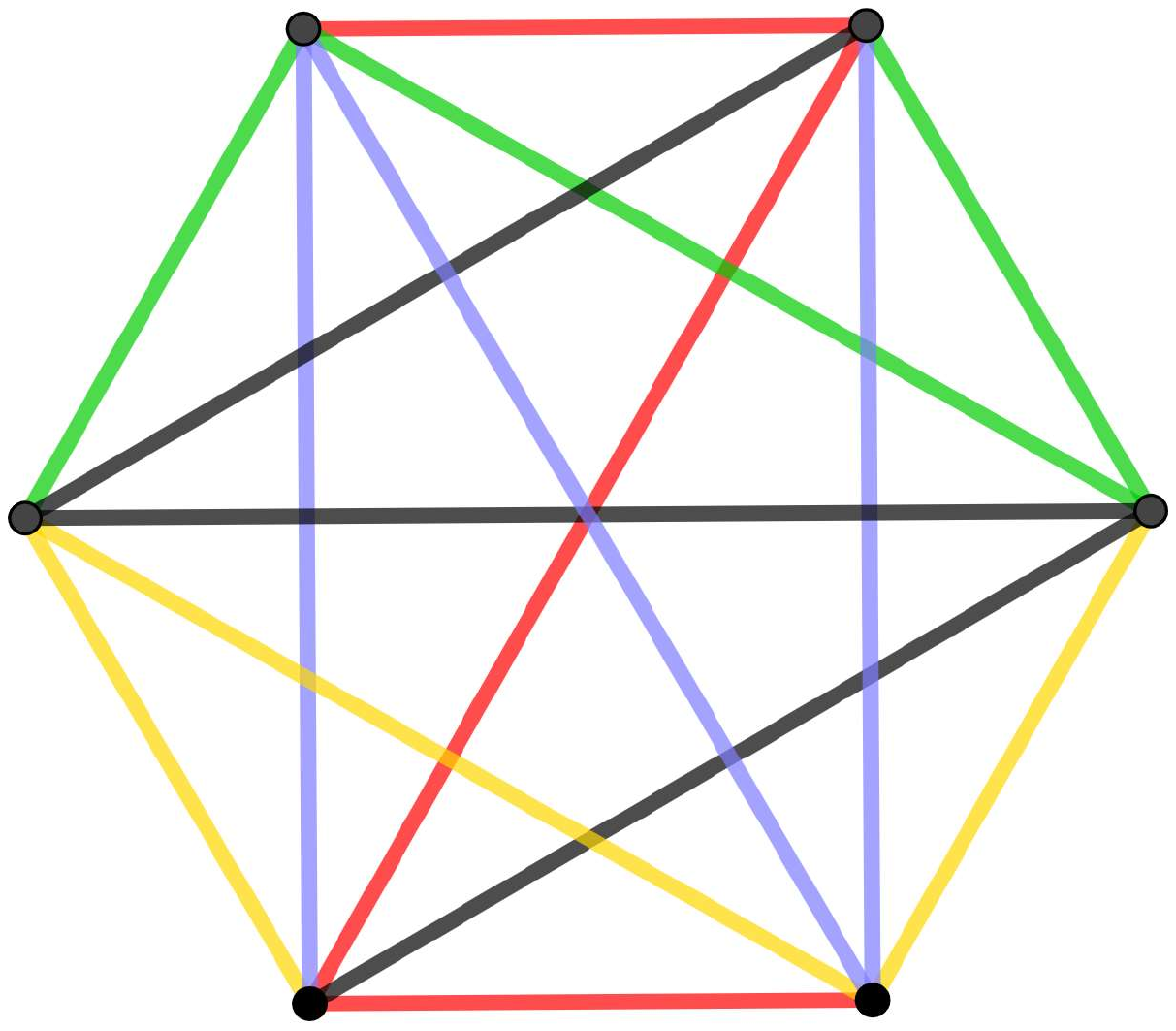}
\includegraphics[width=0.235\textwidth]{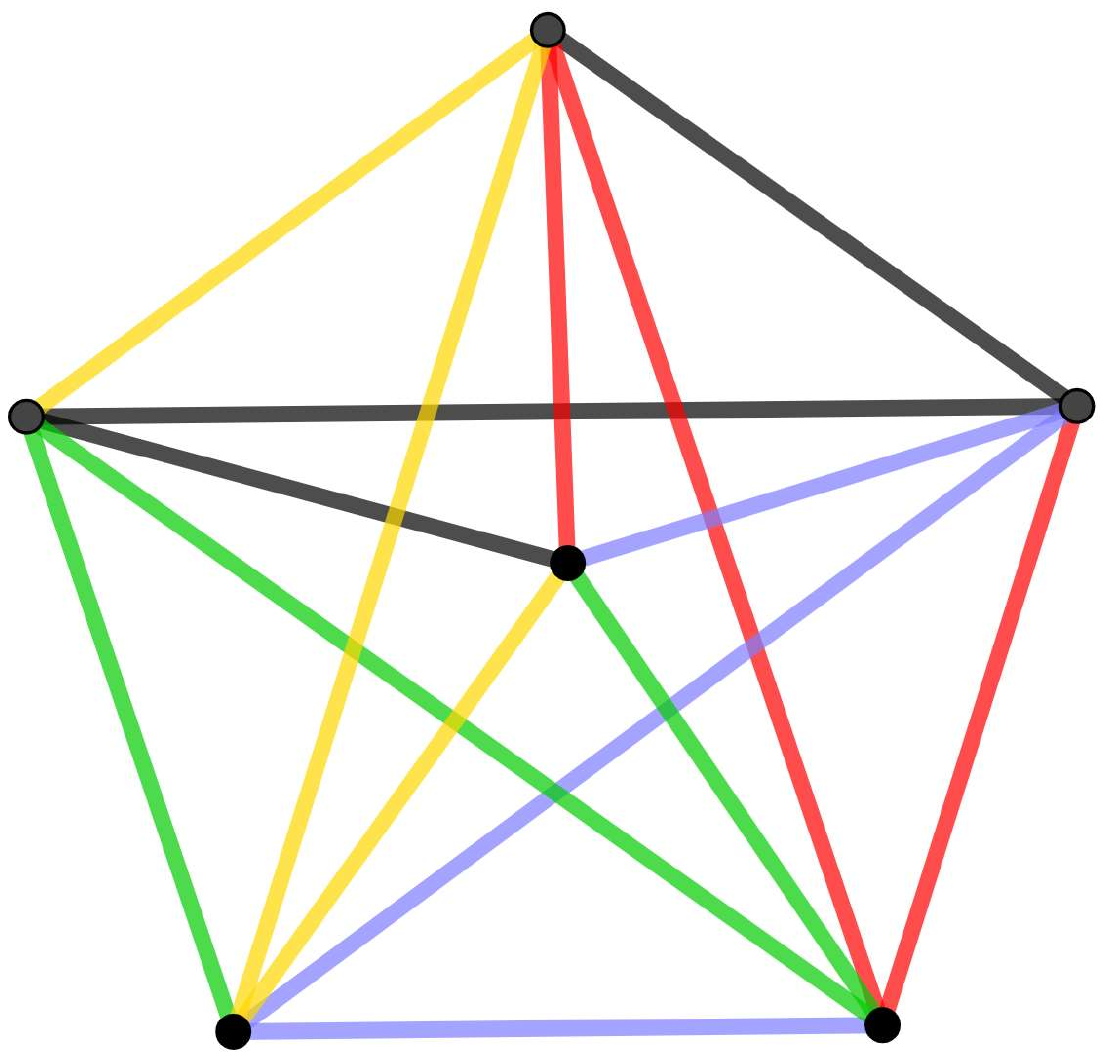}
\includegraphics[width=0.235\textwidth]{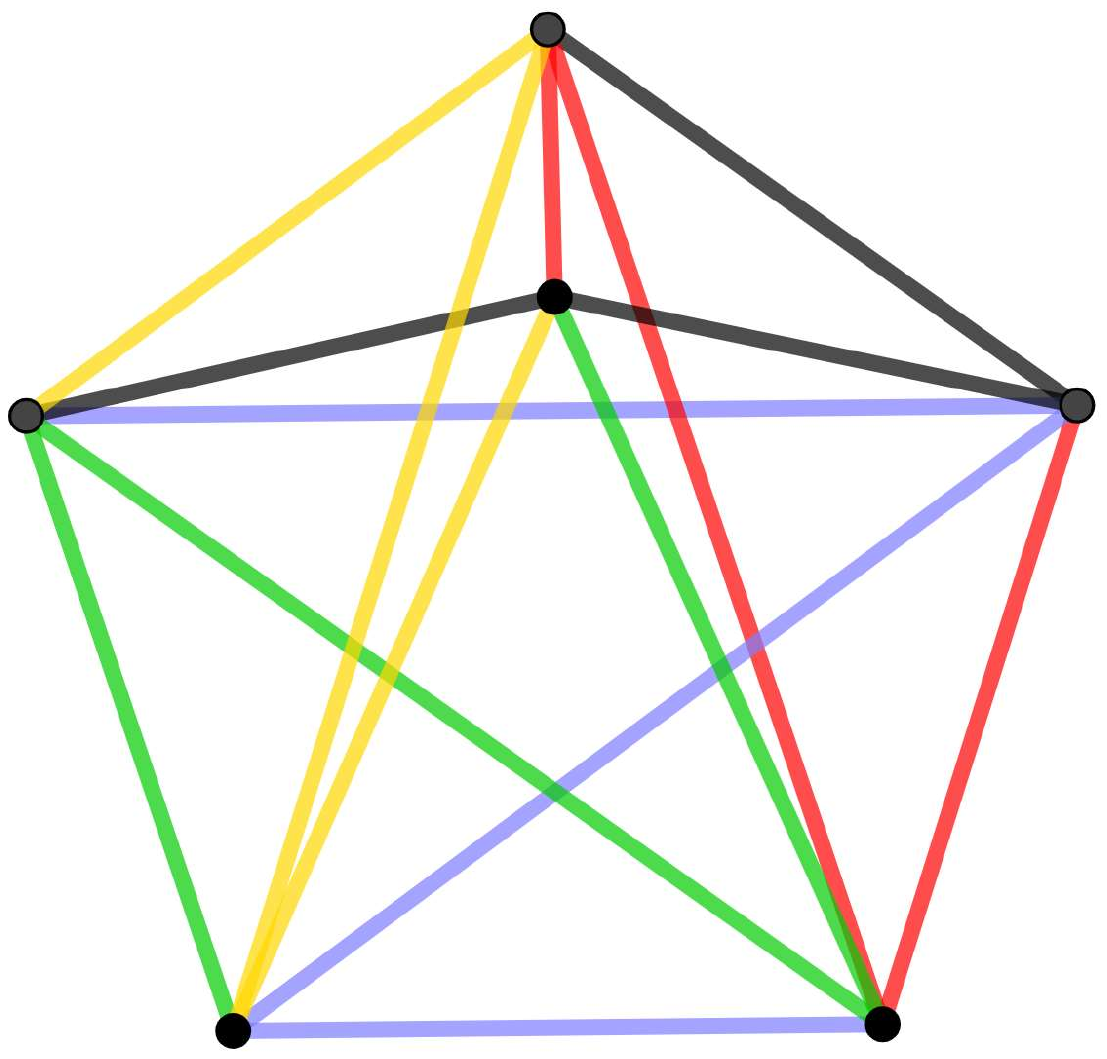}
\includegraphics[width=0.235\textwidth]{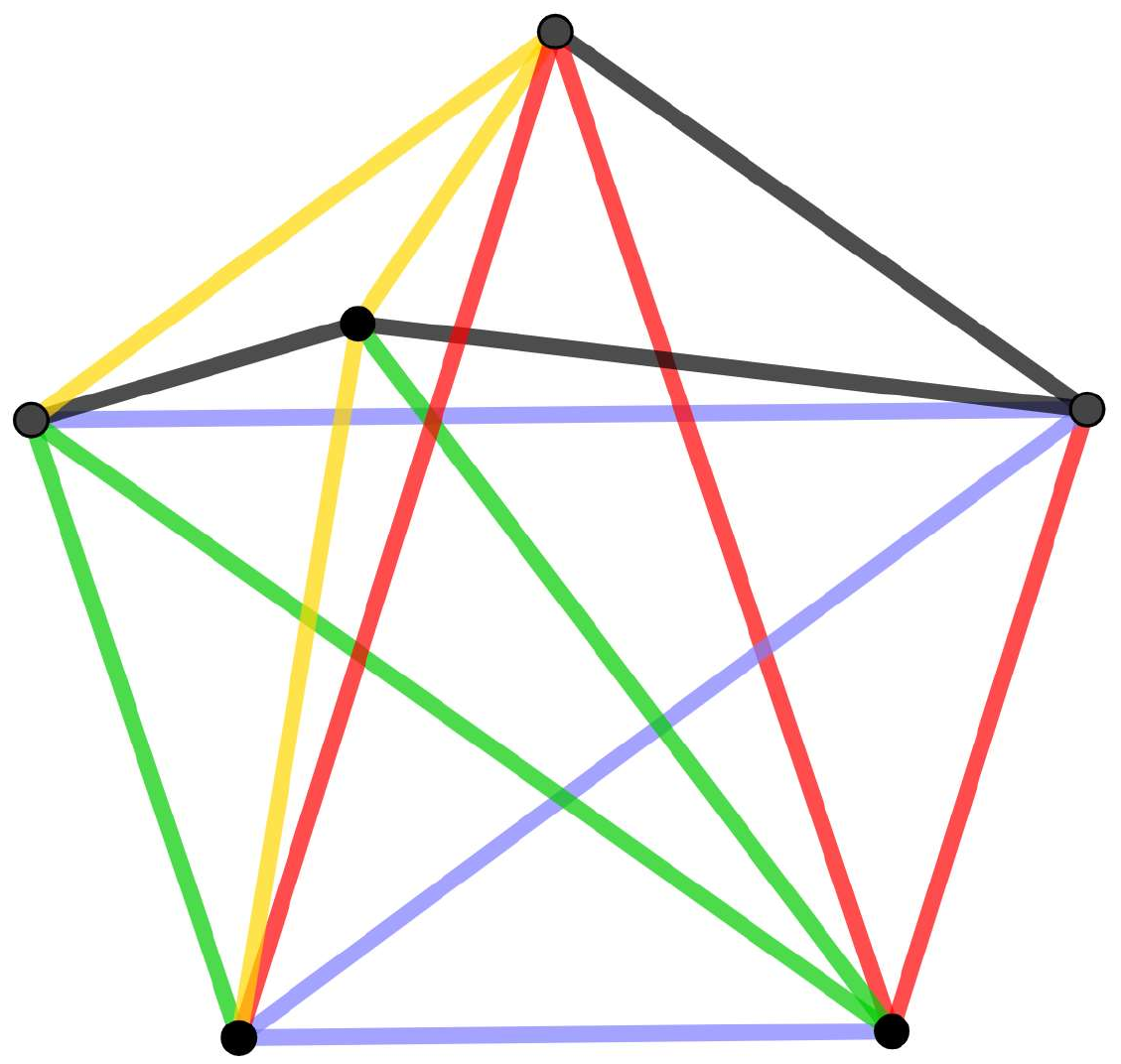}

\includegraphics[width=0.235\textwidth]{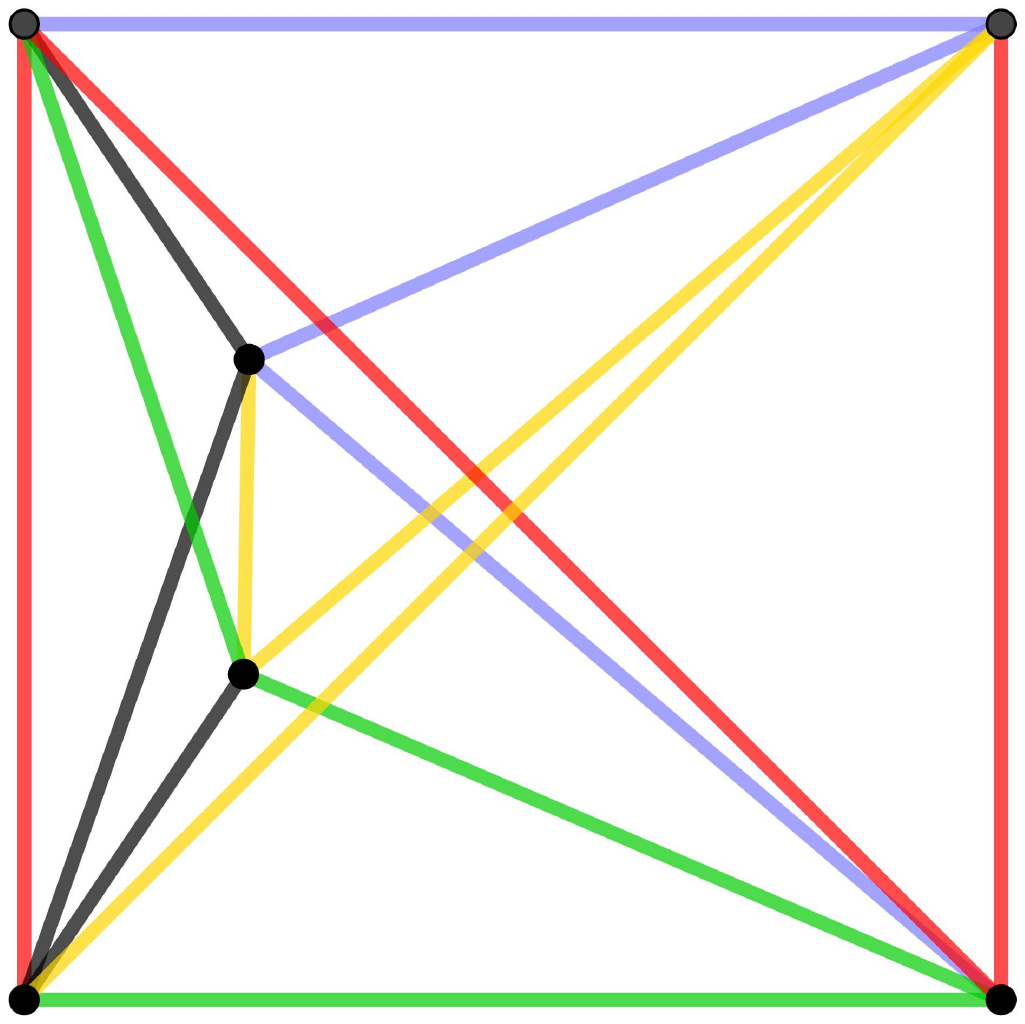}
\includegraphics[width=0.235\textwidth]{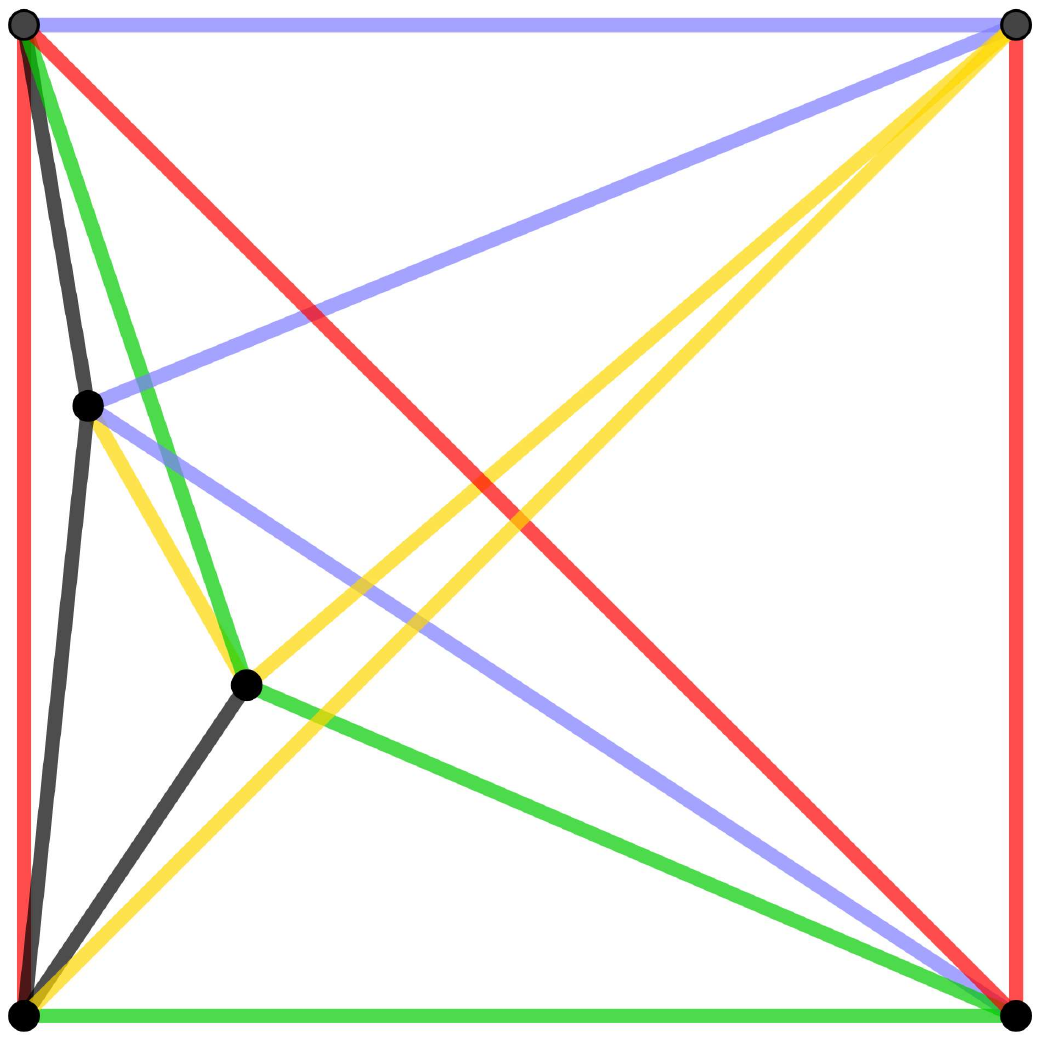}
\includegraphics[width=0.235\textwidth]{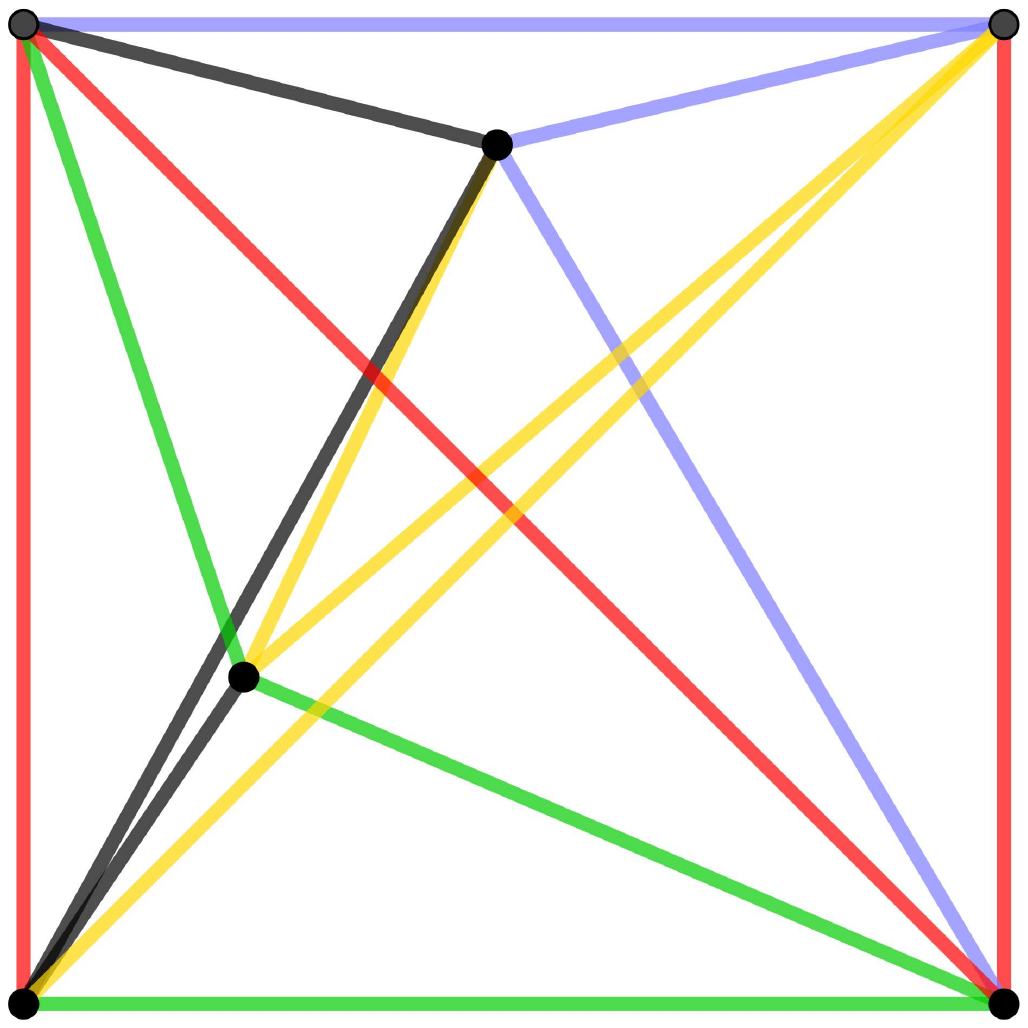}
\includegraphics[width=0.235\textwidth]{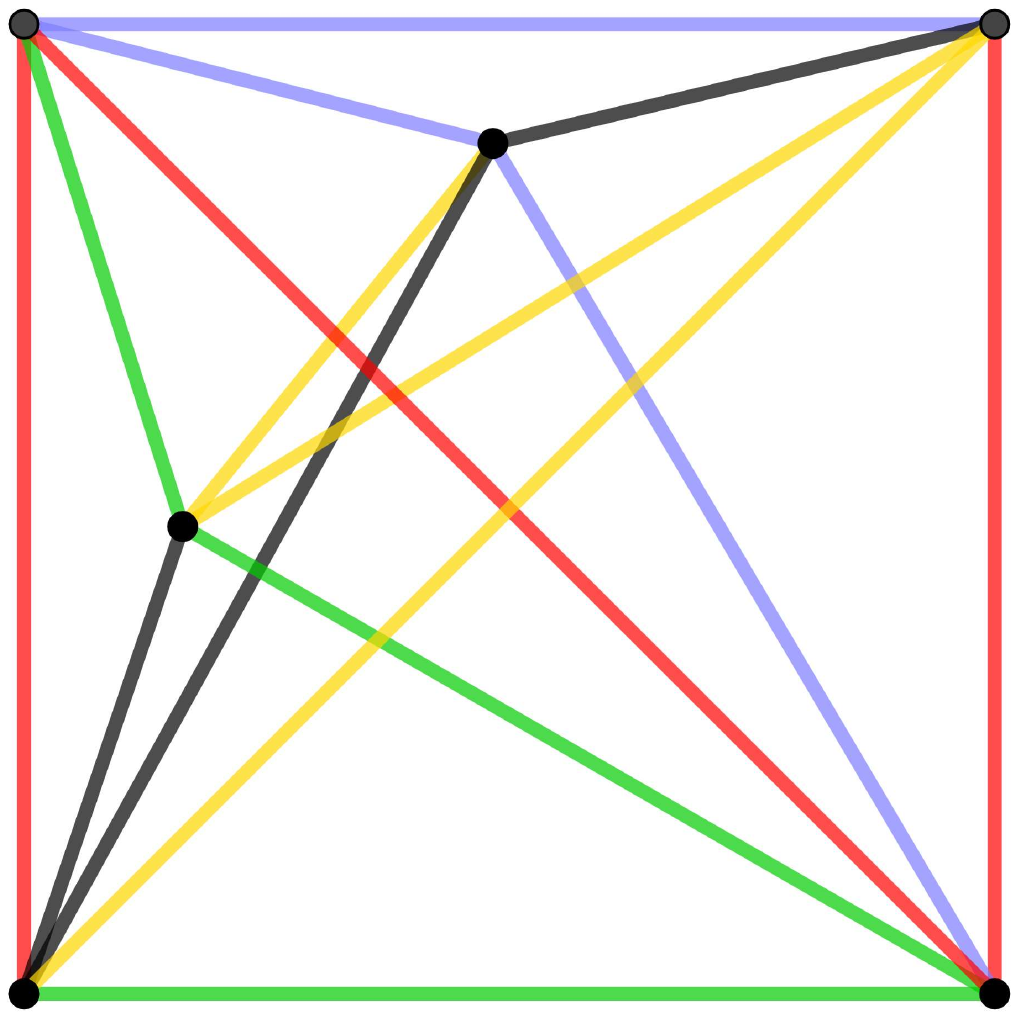}

\includegraphics[width=0.235\textwidth]{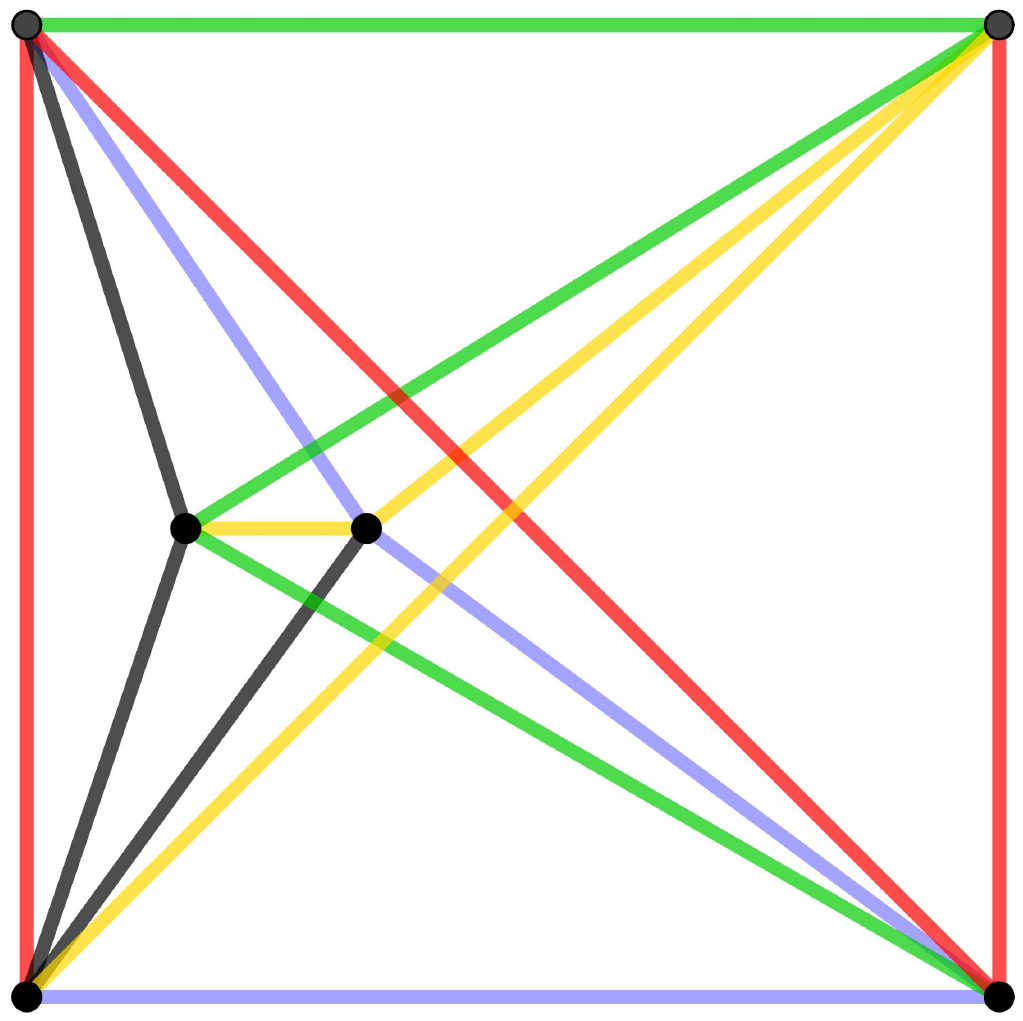}
\includegraphics[width=0.235\textwidth]{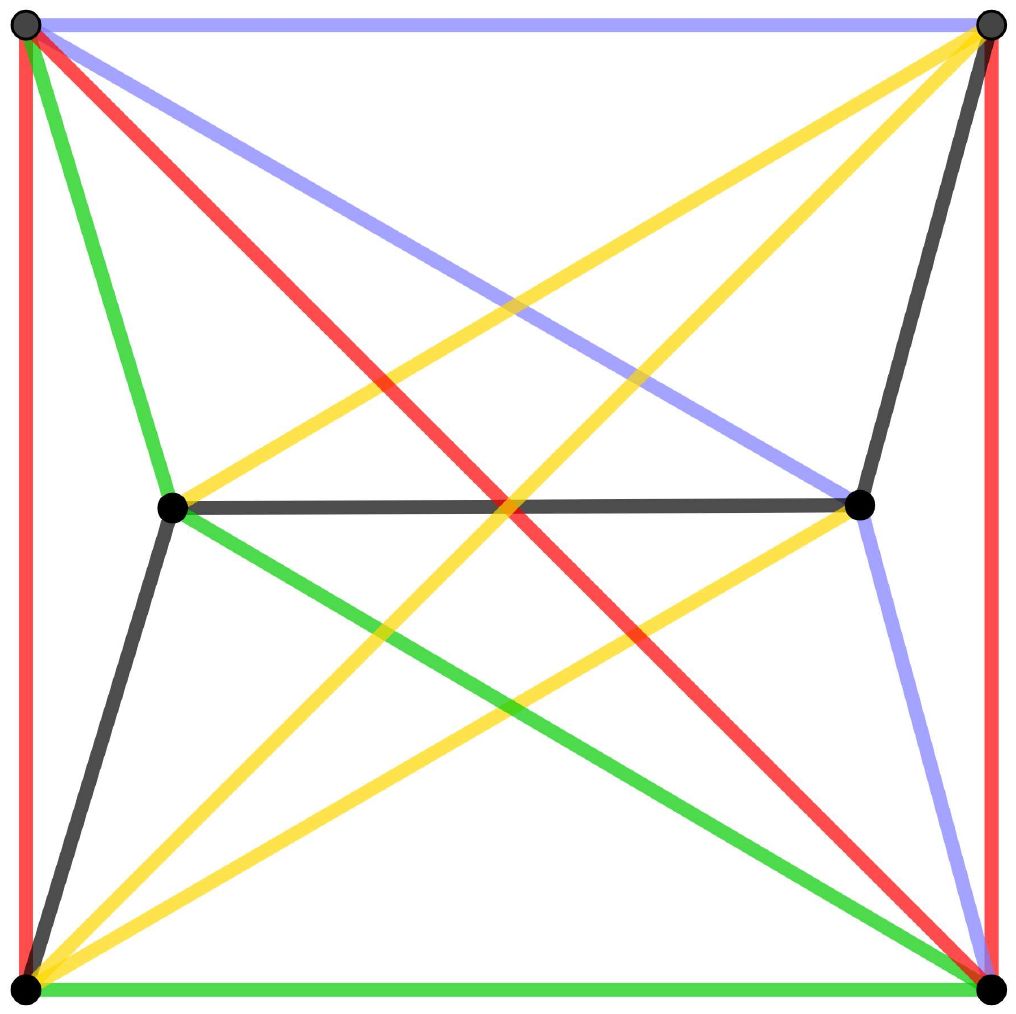}
\includegraphics[width=0.235\textwidth]{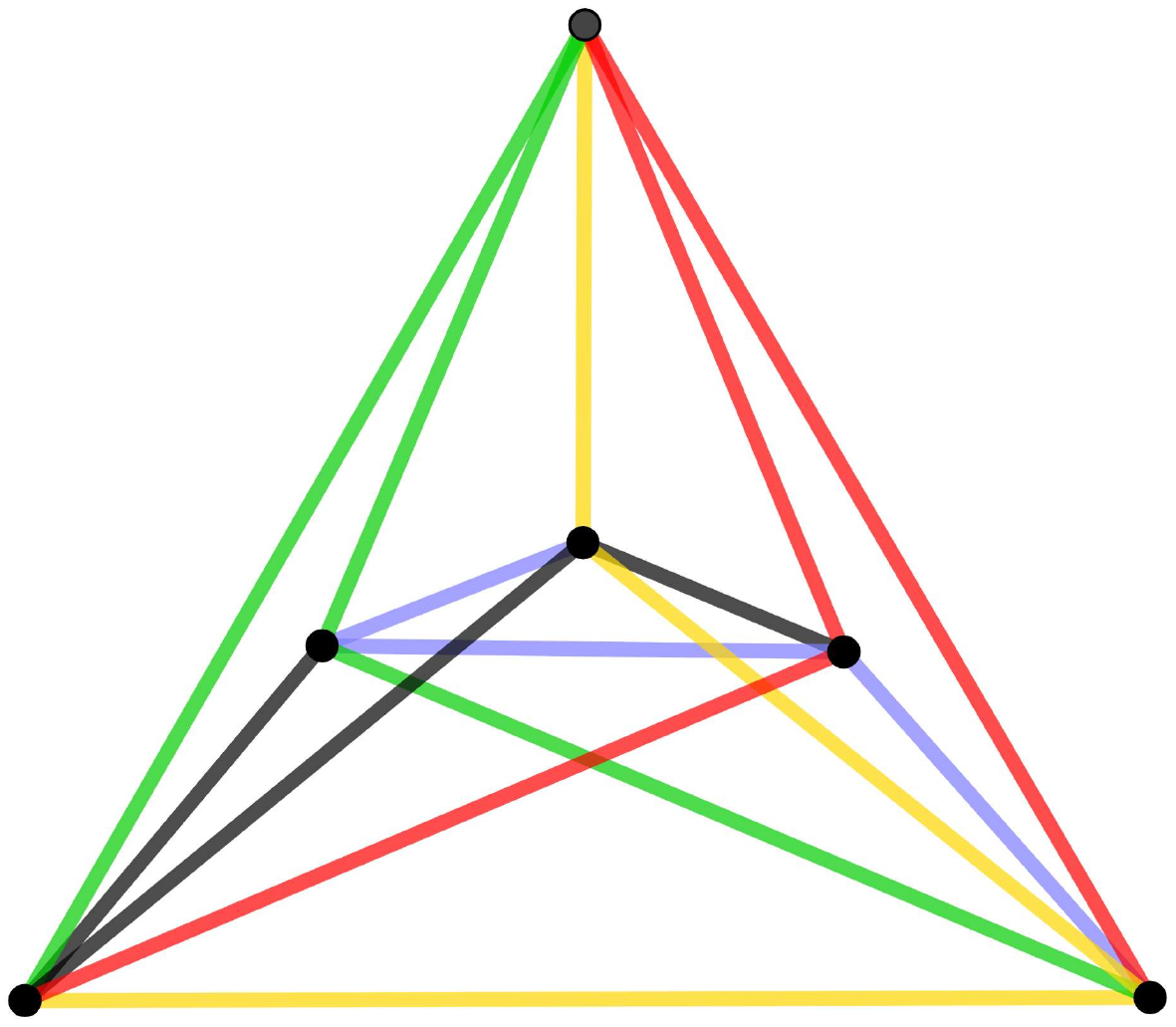}
\includegraphics[width=0.235\textwidth]{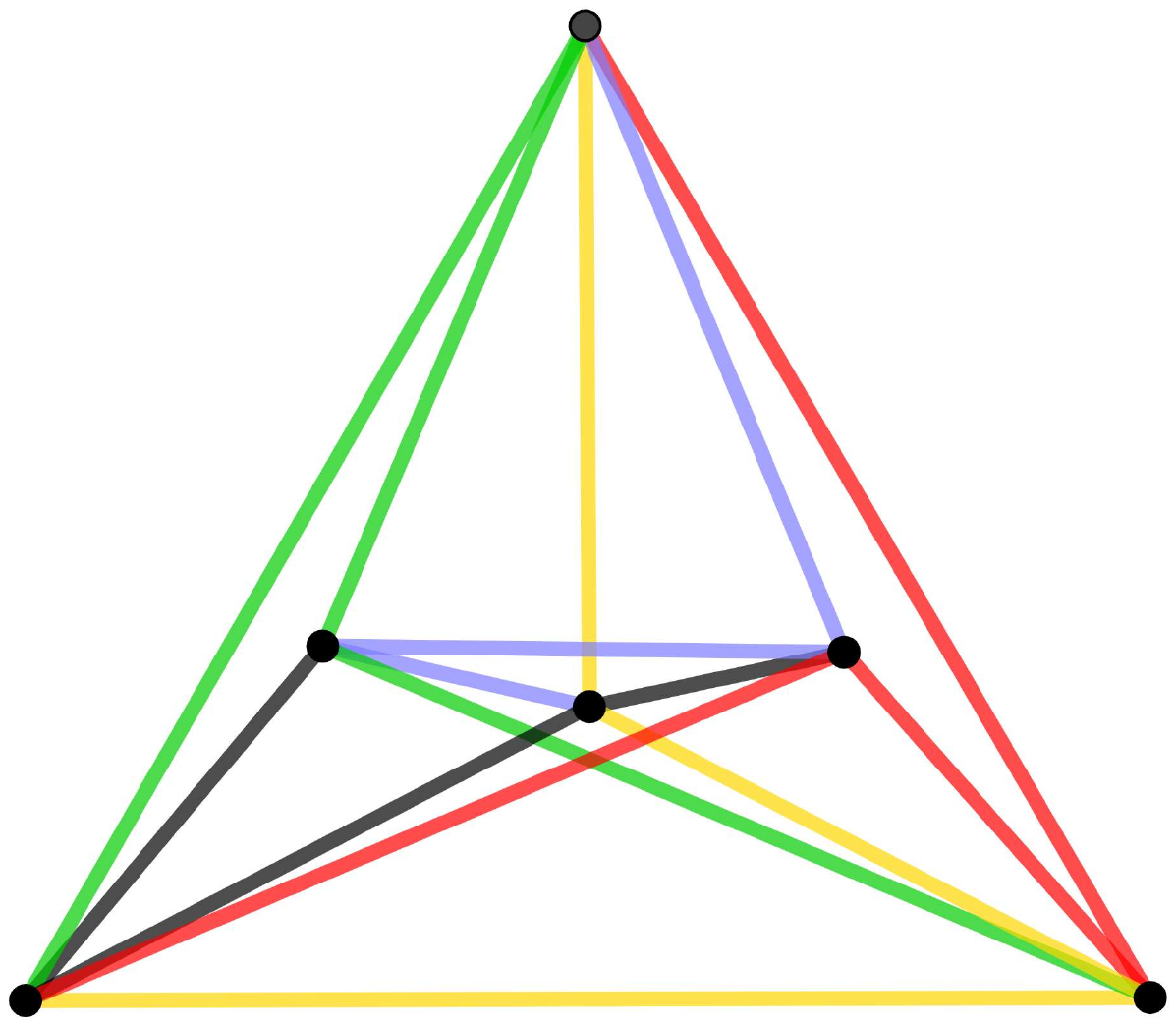}

\includegraphics[width=0.235\textwidth]{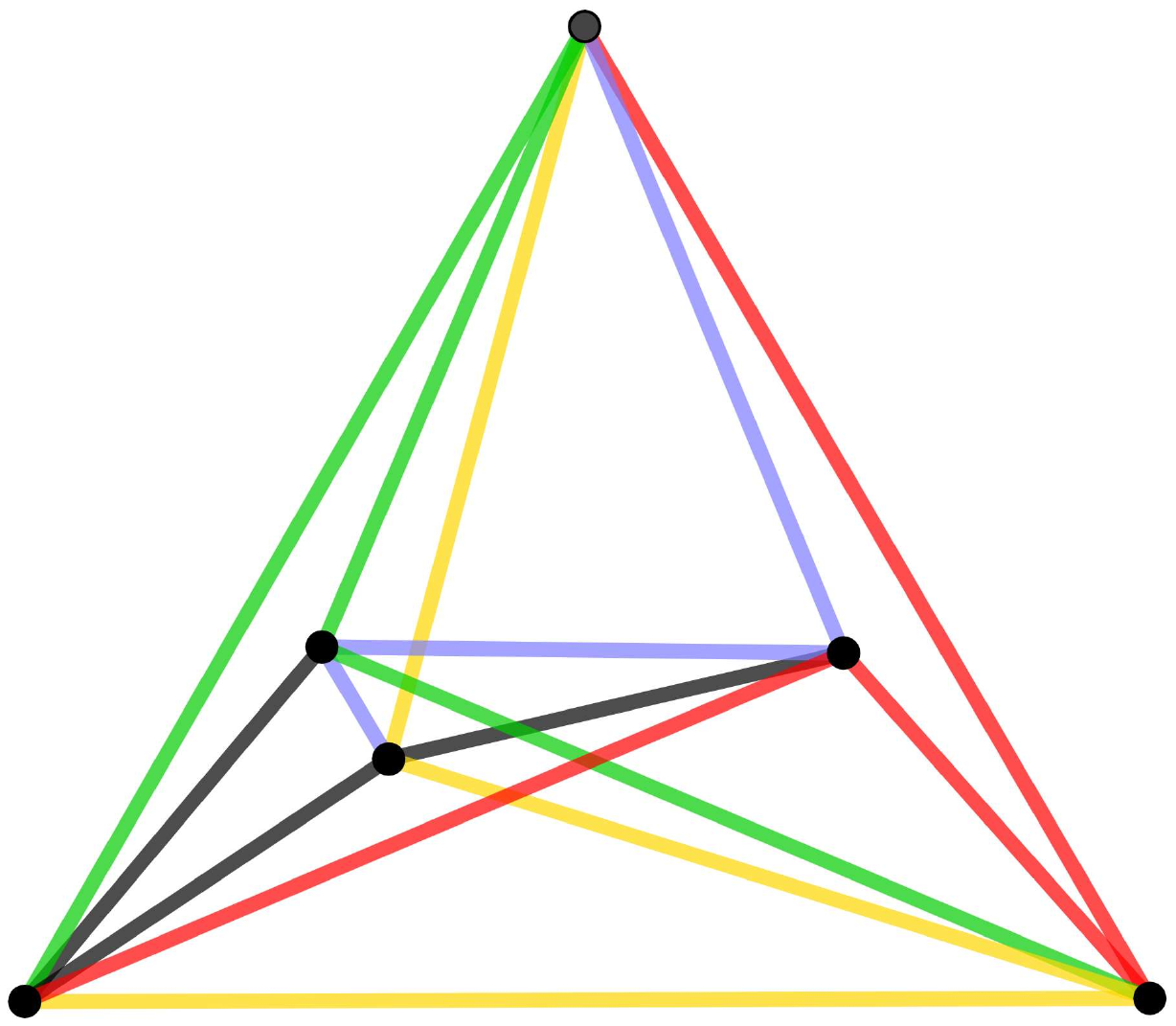}
\includegraphics[width=0.235\textwidth]{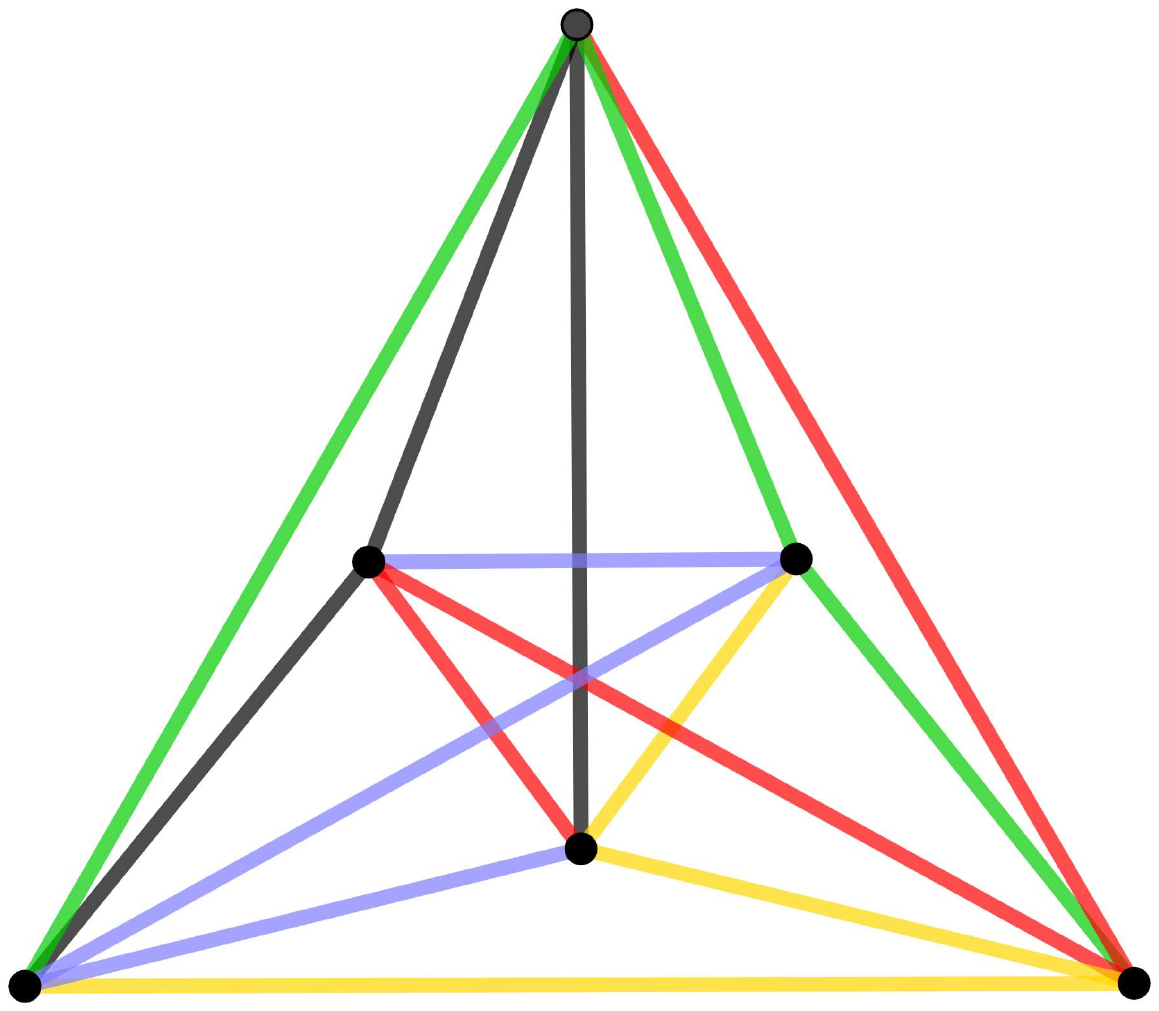}
\includegraphics[width=0.235\textwidth]{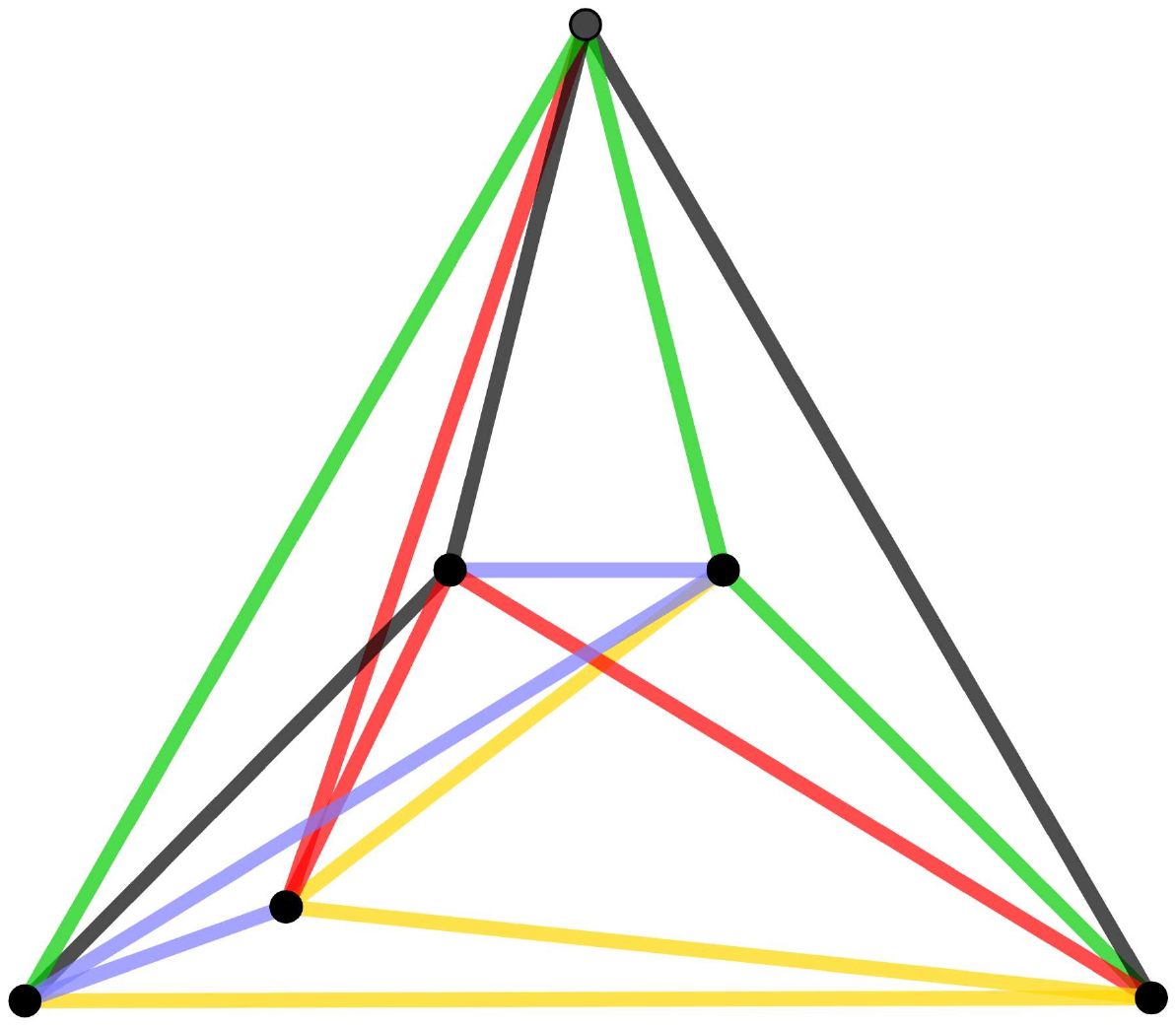}
\includegraphics[width=0.235\textwidth]{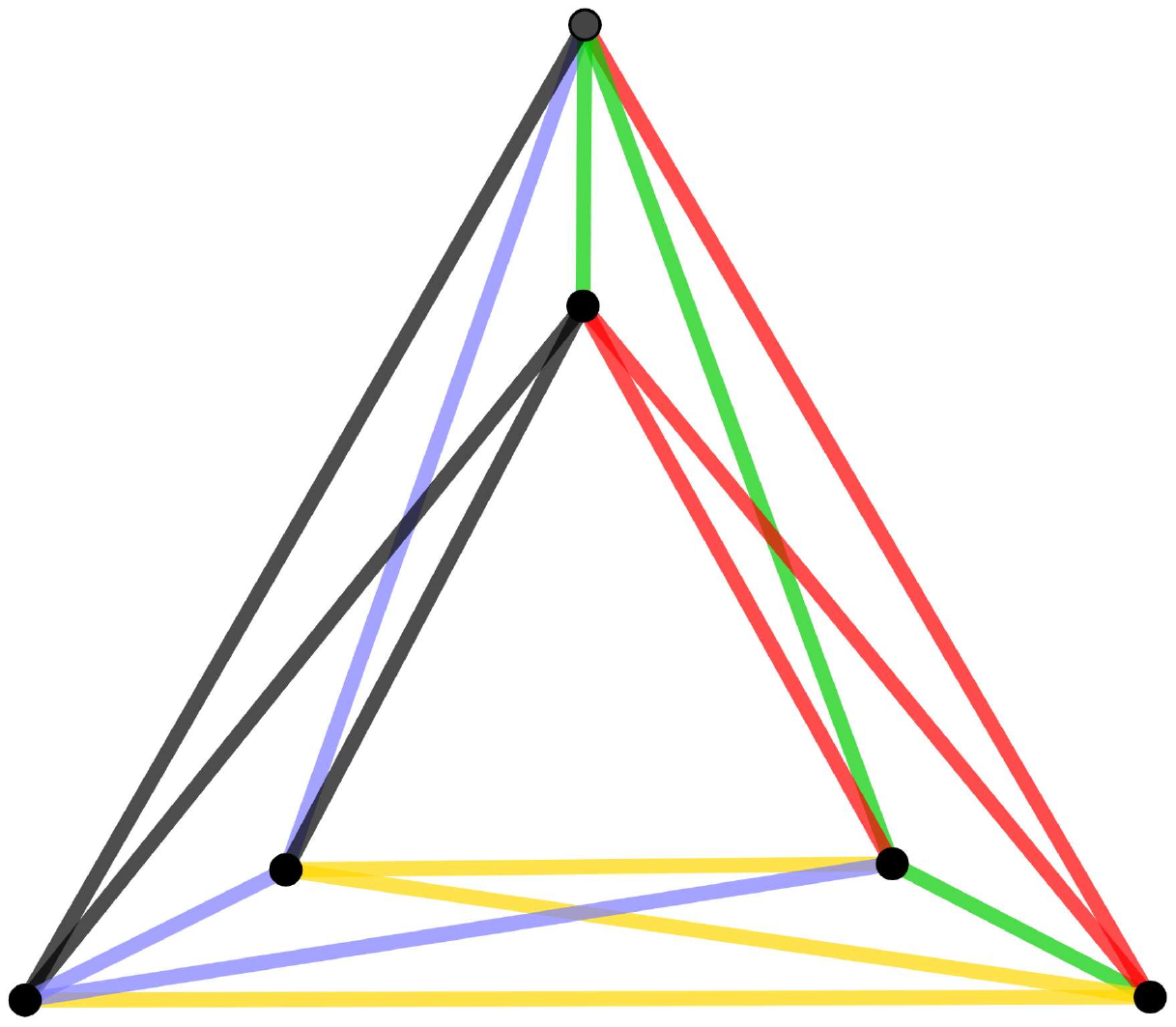}

\caption {Covering the edge set of the complete graph on six points by $3$-edge zig-zag paths.}
\label{fig:16types}
\end{figure}

Our approach is packing as many $K_6$ as possible into $K_n[A]$, then decomposing each $K_6$
into five $3$-edge zig-zag paths, and considering each of the remaining edges as a single path.
This way we can guarantee a decomposition of the edge set of $K_n[A]$ into
$\binom{n}{2}/\binom{6}{2} \times 5 +O(n) = n^2/6 +O(n)$ monotone paths.
It may be worth noting that if $5 \ | \ (n-1)$, the second condition simplifies to $3 \ | \ n(n-1)$.

It suffices to show that for a $6$-point set $A$,
the edge set of $K_6[A]$ can be decomposed into five $3$-edge zig-zag paths.
It is known that there are $16$ order types for such sets, see, \eg, \cite{PW18}.
Since the existence of a zig-zag path decomposition depends only on the order type and not on the specific geometric realization, it suffices to provide zig-zag path decompositions for each type, as we do in Fig.~\ref{fig:16types}.
\qed

\smallskip
A similar bipartite version of the lower bound construction with $n=2k$ points
and $|E_0|=k^2$ (undirected) inter-group edges only yields a lower bound of $n^2/16$: 
If $\P$ is a covering of $K_n[A]$ by monotone paths, it has the property
that every path $\xi \in \P$ contains at most four edges in $E_0$. Consequently,
covering all the edges of the bipartite graph requires at least $k^2/4 = n^2/16$
monotone paths.

\section{Concluding remarks} \label{sec:remarks}

\noindent\textbf{1. } For $k \geq 2$, let $f(k,n)$ denote the maximum number of
edges of a geometric graph on $n$ vertices 
that contains no $k$ pairwise disjoint edges. Early results of
Erd\H{o}s~\cite{Erd46}, Kupitz~\cite{Ku79}, and Alon and Erd\H{o}s~\cite{AE89} 
show that $f(2, n) = n$ and $f(3,n)=O(n)$. 
The first polynomial upper bound for $f(k,n)$, for $k>3$, was established by
Pach and T{\"{o}}r{\H{o}}csik~\cite{PT94}, and was improved by T{\'{o}}th and Valtr~\cite{TV98};
see also~\cite[Chap.~14]{PA95}. 
The current best result, $f(k,n)= O(k^2 n)$ is due to T{\'{o}}th~\cite{To00}. 

It is conjectured that $f(k,n)= O(k n)$~\cite[Chap.~9.5]{BMP05}. If this
conjecture is true, our methods would yield an 
$O(n \log{n})$ (unfortunately, still superlinear) upper bound on the number of
monotone matchings sufficient for covering the edge set of $K_n[A]$,  
analogous to covering this set by crossing-free matchings in~\cite{ADH+05}. 
\smallskip

\noindent\textbf{2. } Obenaus and Orthaber~\cite{OO21} gave a negative answer to
the question of whether every complete 
geometric graph on $n$ vertices ($n$ even) can be partitioned into $n/2$ spanning trees (see~\cite{BHRW06}).
Furthermore, their negative answer extends to the weaker question of whether every complete geometric graph
on $n$ vertices ($n$ even) can be partitioned into $n/2$ plane subgraphs.
See also~\cite{AOO+22}.

On the other hand, it is possible that every complete geometric graph on $n$ vertices 
can be partitioned into $n/2+o(n)$  plane subgraphs. Pach, Saghafian, and Schnider~\cite{PSS23}
proved that a complete convex geometric graph on $n$ vertices cannot be decomposed into fewer than
$n-1$ plane star-forests.

Perhaps the first step towards solving Problem~\ref{problem:1} would be to answer the following question.

\begin{problem} Does there exist a constant $c>1/2$ with the property
  that, for infinitely many values of $n$, there are 
  $n$-element point sets $A$ in the plane such that every covering of the edge set
  of $K_n[A]$ by crossing-free paths requires at least $cn$ paths?  
\end{problem}

\smallskip
\noindent\textbf{3. } Pinchasi and Yerushalmi~\cite{PY24} showed that, given any $n$-element point set $A$
in the plane where $n$ is odd, the edge set of $K_n[A]$ can be partitioned into $(n^2-1)/8$ convex polygons whose vertices belong to $A$. This bound is tight. For the case of even~$n$, $n^2/8 + n/4$
convex polygons suffice and $\lceil n^2/8 \rceil$ are needed. Since every convex polygon can be decomposed
into two monotone paths, we essentially re-obtain the upper bound on monotone paths that comes from covering
by two-edge paths: if $n$ is odd, $(n^2-1)/4$ monotone paths suffice, whereas if $n$ is even,
$n^2/4 + n/2$ monotone paths suffice. These bounds are superseded by our Theorem~\ref{thm:arbitrary}.

\smallskip
\noindent\textbf{4. }
Another interesting class of paths is zig-zag paths. A path is called \emph{zig-zag}
if any three consecutive edges form a $3$-edge zig-zag path, as defined earlier.
Note that this definition does not ensure that the path is monotone or even plane.  

We show that for any finite point set $A$ in general position in the
plane, one can cover the edge set of $K_n[A]$ by plane Hamiltonian
zig-zag paths. Equivalently, for any pair of points $a,b \in A$, the
segment $ab$ appears in at least one such path. To see this, assume
that the line through $a,b$ separates $A\setminus \{a,b\}$ into two
disjoint parts $A_1,A_2$. The edge $ab$ is an edge of the convex hull
of $A_i \cup \{a,b\}$, for $i=1,2$. Starting with $ab$, we construct a
plane zig-zag path that traverses all points in $A_1 \cup \{a,b\}$, as
follows. At the first step, let $b_1 \in A_1$ be the point such that
the angle $\angle abb_1$ is as small as possible. We continue this
process by selecting $b_i$ at the $i$-th step such that the angle
$\angle b_{i-2}b_{i-1}b_i$ is minimized among all remaining points in
$A_1$, where $b_0=b$. It is not hard to verify that once all points
in $A_1$ are included, the resulting path is a plane zig-zag
path. Similarly, starting with $ba$ we can construct a plane zig-zag
path that traverses all points in $A_2 \cup \{a,b\}$. The union of
these two paths is a plane zig-zag Hamiltonian path, as claimed.

While the edge set of any complete geometric graph can be covered by plane
zig-zag (Hamiltonian) paths, the minimum number of  
(not necessarily Hamiltonian) plane zig-zag paths required for covering the edge set of $K_n[A]$ is not known.

\begin{problem}
    What is the smallest number $z=z(n)$ such that the edge set of every
    complete geometric graph on $n$ vertices can be covered by $z$ plane zig-zag paths? 
\end{problem}


\begin{thebibliography}{99}

\bibitem{AOO+22}
Oswin Aichholzer, Johannes Obenaus, Joachim Orthaber, Rosna Paul, Patrick Schnider,
Raphael Steiner, Tim Taubner, and Birgit Vogtenhuber,
Edge partitions of complete geometric graphs,
in \emph{Proc. 38th International Symposium on Computational Geometry (SoCG 2022)},
June 7-10, 2022, Berlin, Germany, LIPIcs series,
Schloss Dagstuhl - Leibniz-Zentrum f{\"{u}}r Informatik,
\textbf{224} (2022), 6:1--6:16, \url{https://doi.org/10.4230/LIPIcs.SoCG.2022.6}.

\bibitem{AE89}
 Noga Alon and Paul Erd{\H o}s,
 Disjoint edges in geometric graphs,
 \emph{Discrete \& Computational Geometry}
\textbf{4} (1989), 287--290.

\bibitem{ADH+05}
Gabriela Araujo, Adrian Dumitrescu, Ferran Hurtado, Marc Noy, and Jorge Urrutia,
On the chromatic number of some geometric type Kneser graphs,
\emph{Computational Geometry: Theory \& Applications} {\bf 32} (2005), 59--69,
\url{https://doi.org/10.1016/j.comgeo.2004.10.003}.

\bibitem{BHRW06}
  Prosenjit Bose, Ferran Hurtado, Eduardo  Rivera-Campo, and David R. Wood,
Partitions of complete geometric graphs into plane trees,
\emph{Computational Geometry}
\textbf{34(2)} (2006), 116--125,
\url{https://doi.org/10.1016/j.comgeo.2005.08.006},
\url{https://doi.org/10.1007/BF0218773}.

\bibitem{BMP05}
Peter Bra\ss , William Moser, and J\'anos Pach,
\emph{Research Problems in Discrete Geometry},
Springer, New York, 2005.

\bibitem{DRT13}
Adrian Dumitrescu, G\"{u}nter Rote, and Csaba D. T\'oth,
Monotone paths in planar convex subdivisions and polytopes,
in {\em Discrete Geometry and Optimization}, 
K. Bezdek, A. Deza, and Y. Ye (editors), 
\emph{Fields Institute Communications 69}, 
Springer, New York, 2013.   

\bibitem {Erd46} Paul Erd\H{o}s,
On sets of distances of $n$ points,
\emph{American Mathematical Monthly}
\textbf{53} (1946), 248--250,
\url{https://doi.org/10.2307/2305092}.

\bibitem {KPT97} Gyula K\'arolyi, J{\'a}nos Pach, and G{\'{e}}za T\'oth,
{Ramsey-type results for geometric graphs. I},
\emph{Discrete and Computational Geometry}
     \textbf{18} (1997), 247-255,
     \url{https://doi.org/10.1007/PL00009317}.

\bibitem {KPTV98} Gyula K\'arolyi, J{\'a}nos Pach,  G{\'{e}}za T\'oth and Pavel Valtr,
{Ramsey-type results for geometric graphs. II},
\emph{Discrete and Computational Geometry}
     \textbf{20} (1998), 375--388,
     \url{ https://doi.org/10.1007/PL00009391}

\bibitem{KLN91}
  Jan Kratochv{\'\i}l,  Anna Lubiw, and Jaroslav Ne{\v{s}}et{\v{r}}il,
Noncrossing subgraphs in topological layouts,
\emph{SIAM Journal on Discrete Mathematics}
\textbf{4(2)} (1991), 223--244,
\url{https://doi.org/10.1137/0404022}.

\bibitem{Ku79}
  Yakov Kupitz,
  \emph{Extremal Problems in Combinatorial Geometry},
  Aarhus University Lecture Notes Series,
  No. 53, Aarhus University, Denmark, 1979.

\bibitem{OO21}
Johannes Obenaus and Joachim Orthaber,
Edge partitions of complete geometric graphs (part~1),
Preprint, 2021, \url{arXiv:2108.05159}.

\bibitem {PA95}
J{\'a}nos Pach and Pankaj Agarwal,
\emph{Combinatorial Geometry},
John Wiley, New York, 1995. 

\bibitem {PSS23}
J{\'a}nos Pach, Morteza Saghafian, and Patrick Schnider,
Decomposition of geometric graphs into star forests,
\emph{Proc. 31th International Symposium on Graph Drawing and Network Visualization (GD 2023)},
vol. 14465 of LNCS, pp.~339--346.
Preprint, 2023, \url{arXiv:2306.13201},
\url{https://doi.org/10.1007/978-3-031-49272-3_23}.

\bibitem {PT94}
J{\'{a}}nos Pach and Jen{\"{o}} T{\"{o}}r{\H{o}}csik,
Some geometric applications of Dilworth's theorem,
\emph{Discrete and Computational Geometry}
\textbf{12} (1994), 1--7,
\url{https://doi.org/10.1007/BF02574361}.

\bibitem {PW18}
Alexander Pilz and Emo Welzl,
Order on order types,
\emph{Discrete and Computational Geometry}
\textbf{59(4)} (2018), 886--922.
\url{https://doi.org/10.1007/s00454-017-9912-9}.


\bibitem {PY24}
Rom Pinchasi and Oren Yerushalmi,
Covering the edge set of a complete geometric graph with convex polygons,
\emph{Discrete and Computational Geometry},
published online \url{https://doi.org/10.1007/s00454-023-00548-3}.

\bibitem {To00}
G{\'{e}}za T{\'{o}}th,
Note on geometric graphs,
\emph{Journal of Combinatorial Theory {A}}
\textbf{89(1)} (2000), 126--132,
\url{https://doi.org/10.1006/jcta.1999.3001}.

\bibitem {TV98}
G{\'{e}}za T{\'{o}}th and Pavel Valtr,
Geometric graphs with few disjoint edges,
\emph{Discrete and Computational Geometry}
\textbf{22(4)} (1998), 633--642,
\url{https://doi.org/10.1007/PL00009482}.

\bibitem {Wil75}
Richard M. Wilson,
An existence theory for pairwise balanced designs, III: Proof of the existence conjectures,
\emph{Journal of Combinatorial Theory, Series A}
\textbf{18(1)} (1975), 71--79,
\url{https://doi.org/10.1016/0097-3165(75)90067-9}.

\end{thebibliography}
\end{document}